\documentclass[11pt]{amsart}

\usepackage{subfiles}
\usepackage{comment}
\usepackage{float}
\usepackage{marginnote}
\usepackage{tabu}
\usepackage{euscript}
\usepackage[dvipsnames]{xcolor}   		             		
\usepackage{graphicx}			
\usepackage{amssymb}
\usepackage{mathrsfs}
\usepackage{amsthm}
\usepackage{amsmath}
\usepackage{stmaryrd}
\usepackage{tikz}
\usepackage{tikz-cd}
\usepackage{accents}
\usepackage{upgreek}
\usepackage{enumerate}
\usepackage{bm}
\usepackage{mathtools}
\usepackage[all]{xy}
\usepackage{caption}
\usepackage{url}
\usepackage{float}
\usepackage{todonotes} 
\usepackage{colonequals}
\usepackage{bbm}
\usepackage{longtable}
\usepackage[full]{textcomp}
\usepackage[cal=cm]{mathalfa}
\usepackage{xparse}
\usepackage{comment}
\usepackage[maxnames=50]{biblatex}

\DeclareDelimFormat{multicitedelim}{\addcomma\space}

\usetikzlibrary{calc}
\usetikzlibrary{fadings}
\usetikzlibrary{decorations.pathmorphing}
\usetikzlibrary{decorations.pathreplacing}
\usepackage{tikz,tikz-cd,tikz-3dplot}
\usepackage{pgfplots}
\usetikzlibrary{arrows,shadows,positioning, calc, decorations.markings, 
hobby,quotes,angles,decorations.pathreplacing,intersections,shapes}
\usepgflibrary{shapes.geometric}
\usetikzlibrary{fillbetween,backgrounds}

\usepackage{colonequals}

\usepackage[margin=1.05in]{geometry}

\tikzset{
  commutative diagrams/.cd, 
  arrow style=tikz, 
  diagrams={>=stealth}
}
\tikzset{
  arrow/.pic={\path[tips,every arrow/.try,->,>=#1] (0,0) -- +(0,4pt);},
  pics/arrow/.default={triangle 90}
}
\tikzset{->-/.style={decoration={
  markings,
  mark=at position .6 with {\arrow{latex}}},postaction={decorate}}
  }
\tikzset{
  c/.style={every coordinate/.try}
}

\theoremstyle{theorem}

\usepackage{hyperref}
\hypersetup{
  colorlinks   = true,          
  urlcolor     = olive,          
  linkcolor    = olive,          
  citecolor   = olive             
}

\makeatletter
\def\@tocline#1#2#3#4#5#6#7{\relax
  \ifnum #1>\c@tocdepth 
  \else
    \par \addpenalty\@secpenalty\addvspace{#2}%
    \begingroup \hyphenpenalty\@M
    \@ifempty{#4}{%
      \@tempdima\csname r@tocindent#1\endcsname\relax
    }{%
      \@tempdima#4\relax
    }%
    \parindent\z@ \leftskip#3\relax \advance\leftskip\@tempdima\relax
    \rightskip\@pnumwidth plus4em \parfillskip-\@pnumwidth
    #5\leavevmode\hskip-\@tempdima
      \ifcase #1
       \or\or \hskip 1em \or \hskip 2em \else \hskip 3em \fi%
      #6\nobreak\relax
    \dotfill\hbox to\@pnumwidth{\@tocpagenum{#7}}\par
    \nobreak
    \endgroup
  \fi}
\makeatother

\newcounter{marginnote}
\DeclareMathAlphabet{\mathpzc}{OT1}{pzc}{m}{it}

\theoremstyle{definition}
\newtheorem{theorem}{Theorem}[section]

\newtheorem{lemma}[theorem]{Lemma}
\newtheorem{proposition}[theorem]{Proposition}
\newtheorem{remark}[theorem]{Remark}

\newtheorem*{runningexample*}{Running example}

\newtheorem*{aside*}{Aside}

\newtheorem{definition}[theorem]{Definition}
\newtheorem{example}[theorem]{Example}

\newtheorem{proposition-definition}[theorem]{Proposition-Definition}

\DeclareMathOperator{\ev}{ev}

\DeclareMathOperator{\Hom}{Hom}

\newcommand{\RR}{\mathbb{R}}

\newcommand{\Gm}{\mathbb{G}_{\operatorname{m}}}

\newcommand{\bcd}{\begin{center}\begin{tikzcd}}
\newcommand{\ecd}{\end{tikzcd}\end{center}}

\newcommand{\PP}{\mathbb{P}}

\newcommand{\N}{\mathbb{N}}
\newcommand{\Z}{\mathbb{Z}}

\newcommand{\A}{\mathbb{A}}

\newcommand{\vir}{\text{\rm vir}}

\newcommand{\calO}{\mathcal{O}}

\newcommand{\CC}{\mathbb{C}}

\newcommand{\LogPk}{\overline{\mathcal{M}}_{0, \upalpha}(\PP^k, \partial \PP^k)}
\newcommand{\LogPkint}{\mathcal{M}_{0, \upalpha}(\PP^k, \partial \PP^k)}
\newcommand{\LogXDint}{\mathcal{M}_{0,\upalpha}(X,D)}
\newcommand{\LogXD}{\overline{\mathcal{M}}_{0, \upalpha}(X,D)}

\begin{document}
 
\title{Upper Bounds for Logarithmic Gromov--Witten Invariants of Projective Space}
\author{Dan Simms}

\begin{abstract}
    We provide an upper bound for the 
    genus zero logarithmic Gromov--Witten invariants of projective space relative
    to its toric boundary. The upper bound is polynomial in the 
    contact orders, with degree depending on the number of marked points.
    The result hinges on positivity of intersections for projective 
    spaces.
\end{abstract}

\maketitle
\setcounter{tocdepth}{1}
\tableofcontents

\section{Introduction}

Consider the space $\mathcal{M}_{g, \upalpha}(X,D)$ of maps 
\[(C, p_1 + \cdots + p_n) \to (X, D_1 + \cdots + D_k)\]
from a smooth genus $g$ curve $C$ to a smooth projective variety
$X$ with simple normal crossings divisor $D = D_1 + \cdots 
+ D_k,$ and fixed tangency orders $\upalpha_{ij}$ between the
marked point $p_i$ and the component $D_j$ of $D$.
Logarithmic Gromov--Witten theory, developed in \cite{GS2013logarithmic, Chen, AbramChen},
enumerates these maps by compactifying 
\[\mathcal{M}_{g, \upalpha}(X,D) \subset 
\overline{\mathcal{M}}_{g, \upalpha}(X,D)\] 
to the space of logarithmic stable maps to $X$, and 
performing intersections against a virtual fundamental class. 
Concretely we use the evaluation maps,
$\ev_i: \overline{\mathcal{M}}_{g,\upalpha}(X,D) \to X,$ 
to define
\[\langle \upgamma_1, \dots, \upgamma_n \rangle^{X \vert D}_{g, \upalpha} \coloneq \int_{[\overline{\mathcal{M}}_{g,\upalpha}
(X,D)]^{\operatorname{vir}}} \prod_i \ev_i^* (\upgamma_i),\]
for $\upgamma_i \in H^*(X).$

In this article, we provide an upper bound for the genus 0 
primary invariants of $\PP^k$ relative its toric boundary
by alternatively compactifying
\[ \LogPkint \subset \PP^{n-3} \times \PP^k.\] 
Here one makes use of an isomorphism
$\LogPkint \cong \mathcal{M}_{0,n} \times T,$ where $T \subset
\PP^k$ is the dense torus. 
The invariants can be expressed as a transverse intersection
in the interior, while the extra contributions from the 
boundary contribute positively because the compactification
is a product of projective spaces. What remains is to perform 
a calculation on this product, whose intersection theory is simple. 

\subsection{Main Result}
Fix a projective space $\PP^k$, some number $n$ of marked points, and
natural numbers $\upalpha_{ij}$ representing the contact order between 
the $i^{th}$ marked point and the $j^{th}$ component of the toric 
boundary of $\PP^k,$ and suppose further that some marked point is 
mapped into the torus of $\PP^k$. That is, there exists $i$ such that
all $\upalpha_{ij} = 0.$

\begin{theorem}\label{Main}
    For each marking $i$, let 
    $\upalpha^{\max}_i = \max_j \upalpha_{ij}.$ Then 
    \[\langle d_1 H^{\upnu_1}, \dots, d_n H^{\upnu_n} 
    \rangle^{\PP^k \vert \partial \PP^k}_{0,\upalpha} \leq 
    d_1 \cdots d_n \binom{n-3+k-\upnu}{n-3}
    \left( \sum_{i=1}^n \upalpha^{\max}_i \right)^{n-3},\]
    where $\upnu$ is the maximal codimension of an insertion at a 
    marked point with no tangency to the boundary. In particular, 
    if one imposes a point constraint at such a point this binomial
    coefficient is 1. 
\end{theorem}

We illustrate the types of upper bounds one obtains with a 
couple of concrete examples.

\begin{itemize}
    \item \textbf{Maximal Contact.} Suppose we count curves which are
    maximally tangent to the toric boundary of $\PP^k$, passing through 
    2 general points. That is, we count 
    curves whose intersection with the $i^{th}$ coordinate hyperplane 
    $H_i$ consists of a single point, necessarily tangent to order 
    $d$ by B\'{e}zout's theorem. Theorem \ref{Main} gives an upper bound 
    for the number of such curves as \[\left((k+1)d\right)^k.\]
    \item \textbf{Characteristic Numbers.} Let $N_d$ denote 
    the number of rational degree $d$ plane curves passing 
    through $3d-1$ general points. These numbers have a storied 
    history, and can be computed recursively \cite{KontsevichManin}.
    It is possible to express $N_d$ as a logarithmic 
    Gromov--Witten invariant as follows. Let $\upalpha = (\upalpha_{ij})_{ij}$
    denote the following tangencies. 
    \begin{table}[H]
        \centering
        \begin{tabular}{c|c c c}
             & $\, H_1 \,$ & $\, H_2 \,$ & $\, H_3 \,$ \\ \hline
             $x_1, \dots, x_d$ & 1 & 0 &  0\\ 
             $y_1, \dots, y_d$ & 0 &  1 & 0 \\ 
             $z_1, \dots, z_d$ & 0 & 0 & 1 \\
             $p_1, \dots, p_{3d-1}$ & 0 & 0 & 0
        \end{tabular}
        \label{CharNumTangencies}
    \end{table}
    The $p_i$ are where the point constraints are imposed. Since they
    are generic points they can be chosen away from the boundary, 
    hence the contact order is zero.
    Meanwhile, $\{x_i, y_i, z_i\}$ represent the intersection points 
    between a generic curve as above with the toric boundary, with 
    generic tangency. The choice of labelling of these
    points introduces a factor of proportionality between the 
    logarithmic Gromov--Witten invariants and the classical count:
    \[\langle 1, \dots, 1, H^2, \dots, H^2 \rangle^{\PP^k \vert 
    \partial \PP^k}_{0, \upalpha}
    = (d!)^3 N_d.\]
    From this, Theorem \ref{Main} yields \[N_d \leq \frac{(3d)^{6d-4}}
    {(d!)^3}.\]
\end{itemize}

\subsection{Proof Outline} \label{outline}

The main point of the proof is that $\LogPkint$ is isomorphic to 
an open inside a product of projective spaces, where all 
contributions to intersection products are positive. The following 
example carries all the information one needs to concretely 
understand the interior $\LogPkint.$

\begin{example}\label{interiorEG}
    Suppose $f: \PP^1 \to \PP^2$ is a linear embedding 
meeting $\partial \PP^2$ transversely at three points. Up to 
automorphisms of $\PP^1$, we can take the three points to be $[0:1], 
[1:1], [1:0].$ Then necessarily $f$ is of the form \[[X:Y] \mapsto 
[\uplambda_0 X: \uplambda_1 (X-Y) : \uplambda_2 Y],\] with $\uplambda_i \in 
\CC^{\times}.$
The space of all such $f$ is therefore abstractly isomorphic to 
$\mathcal{M}_{0,3} \times (\CC^{\times})^2.$ We will see later that we can make 
this isomorphism explicit if there is a marked point mapping into the 
torus of $\PP^2.$
\end{example}
Note that both $\mathcal{M}_{0,n}$ and $(\CC^{\times})^k$ are each isomorphic to open 
subsets of a projective space, $\PP^{n-3}$ and $\PP^k$ respectively.

\begin{itemize}
    \item \textbf{Step 1.} Embed 
    $\LogPkint \hookrightarrow\PP^{n-3} \times \PP^k$ rather than 
    $\LogPkint\hookrightarrow \LogPk.$
\end{itemize}

Since we work in genus 0, there are no virtual considerations necessary.
The invariants are a transverse intersection in $\LogPkint$. That is, 
we can express the invariant as $\deg \prod_i A_i$ for some transverse 
cycles $A_i$ on $\LogPkint.$ One can then take closures $\overline{A_i}
\subset \PP^{n-3} \times \PP^k.$ The product $\prod_i \overline{A_i}$
decomposes as the transverse intersection in $\LogPkint$, whose degree
is the invariant as mentioned, as well as extra components in the 
boundary.

\begin{itemize}
    \item \textbf{Step 2.} Positivity of projective space ensures the extra contributions from the boundary cannot
    be negative. Therefore one obtains 
    \[ \langle \upgamma_1, \dots, \upgamma_n \rangle_{0,
    \upalpha}^{\PP^k \vert \partial \PP^k}
    \leq \deg \prod_i \overline{A_i} .\]
    \item \textbf{Step 3.} Perform a calculation on projective space to 
    determine an upper bound for the degree of $\prod_i 
    \overline{A}_i.$
\end{itemize}

\begin{remark}
    Projective space appears multiple times in this plan, for two 
    different reasons. The choice to compactify $\LogPkint$ with a 
    product of projective spaces is to leverage the positivity of its 
    tangent bundle. On the other hand, we consider maps to projective 
    space because their simple cohomology greatly restricts the possible
    insertions, thereby making explicit calculations feasible. 
\end{remark}

\begin{remark}
    The idea to use positive compactifications to prove bounds on 
    interesting algebro-geometric quantities also appears in 
    contemporaneous work \cite{FNSPositivity, BELL}, the former
    being the main inspiration behind this paper.
\end{remark}

\subsection{Future Work} The methods in this paper apply directly to 
any situation where one performs intersections on rational moduli spaces, or more generally on moduli spaces birational to a variety 
with globally generated tangent bundle. In particular, this method 
will certainly give upper bounds for genus zero logarithmic 
Gromov--Witten invariants of all smooth toric varieties. One expects
that these bounds can be made explicit in nice cases, such as 
products of projective spaces.

\subsection{Acknowledgments} The author gives deep and wholehearted 
thanks to Navid Nabijou for suggesting this project, and for his 
invaluable guidance and care across the years I have known him. 
Without him, this paper would not exist.

This work was supported by the Engineering and Physical 
Sciences Research Council \newline [EP/S021590/1]. The EPSRC Centre for 
Doctoral Training in Geometry and Number Theory (The London 
School of Geometry and Number Theory), University College 
London. The author also thanks Imperial College London.

\section{Background}

\subsection{Logarithmic Gromov--Witten Theory of Toric Varieties}
\label{logGWbackground}
Throughout this section let $(X,D)$ be a smooth projective toric pair, with torus $T \subset X$. 
Let $k = \dim X.$ As usual, we write
$N = \Hom(\Gm, T)$ for the lattice in which the fan $\Upsigma$ associated to $X$ lives, and 
$M = \Hom(T, \Gm).$

\begin{definition}
A choice of \textbf{numerical data} $\uplambda$ is a tuple $(v_1, \dots, v_n) \in N^n$ satisfying \[ \sum_i v_i = 0.\]
This constraint is often called the \textbf{balancing condition}.
\end{definition}
Much information can be extracted from a choice of numerical data 
$\uplambda.$
\begin{itemize}
\item \textbf{Number of marked points.} Take $n = \lvert \uplambda \rvert.$
\item \textbf{Tangency orders.} Since $X$ is smooth and proper, we can write each $v \in N$ uniquely as 
\[v = \sum_{v_{\uprho} \in \Upsigma(1)} \upalpha_{\uprho} v_{\uprho} \]
with $\upalpha_{\uprho} \in \N$ and $v_{\rho}$ primitive. 
Note that $\upalpha_{\uprho} \neq 0$ if and only if $\uprho$ is a 
generator for the minimal cone of $\Upsigma$ containing $v.$
Each $v_{\rho}$ corresponds to a component of the toric boundary 
$D$. Decomposing $v_i$ in this way produces a tuple $\upalpha_{ij},$ 
which we take to be tangency orders between the $i^{th}$ marking and 
$D_j$.
\item \textbf{Curve class.} Denote by $\upbeta$ the curve class of a map
$[f:C \to X].$ An operational Chow class on a toric variety is determined
by its pairings $\upbeta \cdot D_j$ with the components of the boundary via the theory of Minkowski weights
\cite[Theorem 2.1]{FultonSturmfels}. The balancing condition ensures 
that there is a curve class $\upbeta$ with $\upbeta \cdot D_j = \sum_i \upalpha_{ij}$ for all
$j$. This statement is simply the exactness of 
\[0 \to A_1(X) \to \Z^{\Upsigma(1)} \to N \to 0,\]
dual to the familiar \[0 \to M \to \Z^{\Upsigma(1)} \to 
\mathrm{Cl}(X) \to 0.\]
\end{itemize}

\begin{remark}
    One advantage of working with tangencies as elements of $N$ is the
    independence of the choice of toric variety compactifying the torus.
    This is beneficial, for example, if one obtains $X'$ from $X$ by 
    subdividing its fan. Then a choice of numerical data for $X$ as we 
    define it immediately lifts along the natural map $X' \to X$ to give 
    valid numerical data for a map to $X'$. This is much easier than 
    choosing a lift of a curve class and checking compatibilities.
\end{remark}

The only other invariant of stable maps constant in flat families is the 
genus. We record this separately.

\begin{definition}
    Denote by \[\LogXD\] the space of genus zero logarithmic stable 
    maps to $(X,D)$ with tangency orders $\upalpha$ 
    \cite{GS2013logarithmic, Chen, AbramChen}. This carries a 
    natural logarithmic structure; denote by 
    $$\LogXDint$$ the locus where the log structure is trivial. This 
    parametrises the maps from smooth curves, whose image 
    meets $D$ precisely at the points $p_i$ with the prescribed 
    tangencies. 
\end{definition}

\begin{proposition}(\cite[Proposition 3.3.3]
{Ranganathan_2017})
\label{intExplicit}
    Fix numerical data $\upalpha = (v_1, \dots, v_n),$ with $v_i=0$.
    Then the map \[ \mathrm{fgt} \times \ev_i: \LogXDint 
    \to \mathcal{M}_{0,n} \times T\] is an isomorphism, where 
    \[\mathrm{fgt}: \LogXDint \to \mathcal{M}_{0,n}\]
    is the forgetful map sending a map to its, necessarily 
    smooth, marked source curve. 
\end{proposition}

The proof follows exactly as Example \ref{interiorEG}, appealing in 
general to the homogeneous coordinate ring on a smooth toric variety 
\cite{CoxHomog}.
As noted in \cite{Ranganathan_2017}, there is an abstract isomorphism 
even without the assumption that some $v_i = 0,$ but we rely on an 
explicit construction of the isomorphism. 

While virtual machinery is in general 
necessary to compute invariants, the case of genus zero maps to a toric
variety is simpler.

\begin{theorem} The space $\LogPk$ is irreducible of the expected
dimension. Moreover,
\[ [\LogPk]^{\operatorname{vir}} = [\LogPk].\]
\end{theorem}

\begin{proof}
    The irreducibility of $\LogPk$ is \cite[Proposition 3.3.5]
    {Ranganathan_2017}.
    A relative perfect obstruction theory for $\LogPk \to 
    \mathfrak{M},$ where $\mathfrak{M} = \overline{\mathcal{M}}_{0,\alpha}(\mathcal{A}_{\PP^k \vert \partial \PP^k})$ is the space of 
    maps to the Artin fan of $\PP^k,$ is 
    constructed in \cite{AbramovichWise} as follows.
    Denote by $\pi: \mathcal{C} \to \LogPk$ the universal curve,
    and $f: \mathcal{C} \to \PP^k$ the universal map. Then
    \[E^{\bullet} = \mathrm{R}\pi_* f^* T^{\log}_{\PP^k \vert 
    \partial \PP^k} \to \mathbb{L}_{\LogPk/\mathfrak{M}}\]
    is a relative perfect obstruction theory.
    But $T^{\log}_{\PP^k \vert \partial \PP^k} \cong \calO_{\PP^k}^{\oplus k}
    ,$ so in genus zero there are no obstructions. Therefore 
    $\LogPk$ is smooth over $\mathfrak{M},$ and
    $[\LogPk]^{\vir} = [\LogPk].$ 
\end{proof}

Given cohomology classes $\upgamma_1, \dots, \upgamma_n \in 
H^*(\PP^k),$ one defines logarithmic Gromov--Witten invariants 
\[\langle \upgamma_1, \dots, \upgamma_n \rangle^{\PP^k \vert
\partial \PP^k}_{0, \alpha} \coloneq \int_{[\LogPk]} \prod_i
 \ev_i^* \upgamma_i.\]

\subsection{Positivity} \label{positivitySec}

The other key ingredient we need is the following theorem on positivity 
of intersections.

\begin{theorem}{(\cite[Corollary 12.2]{FultonIntersection})} 
\label{positivityThm}
    Let $X$ be a nonsingular variety with globally generated tangent 
    bundle, and let $V_1, \dots, V_k$ be subvarieties with $\sum 
    \mathrm{codim}(V_i/X) = \dim X.$ Let $Z_i$ be the 
    distinguished varieties of the scheme-theoretic
    intersection $\bigcap V_i$, and $\upzeta_i$ be the component of 
    $\prod_j [V_j]$ in $A_0(Z_i).$ Then $\deg \upzeta_i \geq 0$ for 
    all $i$.
\end{theorem}

While transverse intersections are positive in complex geometry, due to 
canonical orientations, negativity can be introduced from excess 
intersections. Intuitively, this arises from negativity of a 
normal bundle. But all such arise as quotients of the ambient
tangent bundle, so a positivity assumption on the tangent 
bundle provides a uniform way of obstructing
negative contributions. 

In our applications $X$ will be a product of projective spaces,
which satisfies the hypothesis because of the Euler sequence.

\section[Proof of Theorem 1.1]{Proof of Theorem \ref{Main}}

We consider genus zero logarithmic maps to $\PP^k$ relative to its 
toric boundary, and insist that at least one marked point is mapped into 
the torus $T \subset \PP^k$. Throughout this text, 
we will fix a number $n$ of marked points, labelled $\{ p_1, 
\dots, p_n \}.$ Fix also a labelling $H_0, \dots, 
H_k$ of the boundary components, and tangency data $\upalpha_{ij}$ 
between the marked point $p_i$ and the boundary component $H_j.$
We assume that at least one marked point
is not tangent to any component of $\partial \PP^k.$ Any map 
with these tangency data is necessarily of degree $d = \sum_i
\upalpha_{ij}$ for any $j$.

Note that we can use multilinearity of Gromov--Witten invariants to 
restrict to linear insertions. It is therefore more convenient to 
work with the following definition. 

\begin{definition}
    Fix a partition $\upnu = (\upnu_1, \dots, \upnu_n)$ of $n+k-3$. Define 
    \[N_{\upalpha, \upnu} \coloneq \langle H^{\upnu_1}, \dots, H^{\upnu_n} 
    \rangle^{\PP^k \vert \partial \PP^k}_{0, \upalpha},\]
    where $H = c_1(\mathcal{O}_{\PP^k}(1))$ is the hyperplane class.
\end{definition}

\subsection{Transversality} \label{transversality}

In Sections \ref{transversality} and \ref{AltComp} we establish the theoretical underpinnings of the upper
bound, before moving onto explicit calculations in Section 
\ref{calculations}. We express $N_{\upalpha, \upnu}$ as a transverse intersection in $\LogPkint.$

\begin{proposition} \label{invIsPtCount}
    Fix tangency orders $\upalpha$ and insertions $\upnu.$
    Then choosing generic linear subspaces 
    $V_i \cong \PP^{\upnu_i}$ in $\PP^k,$ the scheme
    \[ \bigcap_{i=1}^n \ev_i^{-1} V_i \]
    consists of $N_{\upalpha, \upnu}$ many reduced points contained 
    in $\LogPkint.$
\end{proposition}

\begin{proof}
This follows from Kleiman's theorem \cite{Kleiman}. Indeed, 
$\prod_i \operatorname{PGL}_{k+1}$ acts transitively on $\prod_i
\PP^k.$ Having chosen $H_i$ generically, define $A$ and $B$ 
via the following Cartesian diagram
\[\begin{tikzcd}
	A & B & {\prod_iV_i} \\
	{\partial \LogPk} & \LogPk & {\prod_i \PP^k}
	\arrow[from=1-1, to=1-2]
	\arrow[from=1-1, to=2-1]
	\arrow[from=1-2, to=1-3]
	\arrow[from=1-2, to=2-2]
	\arrow[from=1-3, to=2-3]
	\arrow[from=2-1, to=2-2]
    \arrow[from=1-1, to=2-2, phantom, "\square" xshift={17pt}]
    \arrow[from=1-2, to=2-3, phantom, "\square" xshift={7pt}]
	\arrow["{\prod_i \ev_i}", from=2-2, to=2-3]
\end{tikzcd}\]
We deduce that $B = \bigcap_{i=1}^n \ev_i^{-1} V_i$ consists of 
finitely many reduced points, and further that $A = B \cap 
\partial \LogPk$ is empty. So all points of $B$ 
are contained in the interior $\LogPkint$ of the 
moduli space. That there are $N_{\upalpha, \nu}$ many points 
follows from \cite[Proposition 8.2(b)]{FultonIntersection}, 
noting that the smoothness hypotheses are satisfied.
\end{proof}

\subsection{Alternative Compactification} \label{AltComp}

We know from Proposition \ref{intExplicit} that a choice of 
marked point with no tangency to $\partial \PP^k$ gives an 
explicit isomorphism $\LogPkint \cong 
\mathcal{M}_{0,n} \times T.$ Moreover, a choice of any three 
marked points rigidifies the action of $\operatorname{PGL}_2$
on $\PP^1$ and gives an isomorphism between
$\mathcal{M}_{0,n}$ and an open subset of $\A^{n-3}.$ We make
these choices now. 

Fix a labelling $(p_1, \dots, p_n)$ of the marked points, and send $(p_1,p_2,p_3) \mapsto (0,1,\infty).$ We will also 
denote by \[p = p_{i_0} \in \{p_4, \dots, p_n \} \] a fixed choice of marked point with no tangency to the 
boundary. In homogeneous coordinates on $\PP^1,$ write $p_i = [S_i: T_i].$
Our rigidification of the $\operatorname{PGL}_2$ action was achieved by declaring $p_1 = [0:1], p_2 = [1:1], p_3 = [1:0].$ 
We obtain the following affine coordinates on $\mathcal{M}_{0,n} \subset \A^{n-3}$:  
\[x_4 = \frac{S_4}{T_4}, \dots, x_n = \frac{S_n}{T_n}.\]
Compactifying to $\PP^{n-3}$ gives homogeneous coordinates $X_0, X_4, \dots, X_n$ with 
\begin{equation}\label{eq1}
\frac{X_i}{X_0} = x_i = \frac{S_i}{T_i}.\end{equation} 

Also naturally $T \subset \A^{k},$ say with coordinates $a_i = A_i/A_0,$ where $A_0, \dots, A_k$ are homogeneous 
coordinates on $\PP^k.$ Now $\LogPkint$ is isomorphic to an open subset of $\A^{n-3} \times \A^k,$ or equivalently an open 
subset of $\PP^{n-3} \times \PP^k.$

Section \ref{transversality} relates $\bigcap_i \ev_i^{-1} V_i,$ for generic linear subspaces
$V_i$ of $\PP^k,$ to logarithmic Gromov--Witten invariants. The positivity considerations of Section 
\ref{positivitySec} lead us to consider the intersection of the closures $\bigcap_i \overline{ \ev_i^{-1} V_i}$ in $\PP^{n-3}
\times \PP^k.$

\begin{proposition}\label{mainTheory}
    Denoting by $H_{ij},$ for $i = 1, \dots, n,$ 
    $j = 1, \dots, \upnu_i,$ generic hyperplanes in $\PP^k,$ 
    we have an upper bound
    \[ N_{\upalpha, \upnu} \leq \deg \prod_{i=1}^n \prod_{j=1}
    ^{\upnu_i}
    [\overline{\ev_i^{-1}H_{ij}}].\]
\end{proposition}

\begin{proof}
    Consider the intersection $\bigcap \overline{\ev_i^{-1}V_i}$
    in $\PP^{n-3} \times \PP^k,$ with $V_i \cong \PP^{\nu_i}$
    generic linear subspaces of $\PP^k.$
    Denote by $Z_i$ the points of the intersection
    contained in $\LogPkint,$ and by $W_i$ the distinguished
    varieties
    contained in the boundary $(\PP^{n-3} \times \PP^k)
    \setminus \LogPkint.$
    Write \[ \prod_{i=1}^n \prod_{j=1}
    ^{\upnu_i}
    [\overline{\ev_i^{-1}H_{ij}}] = z + w \in 
    \bigoplus A_0(Z_i) \oplus A_0(\cup_i W_i).\]
    Theorem \ref{positivityThm} tells us that $\deg w \geq 0,$
    while $\deg z = N_{\upalpha, \upnu}$ by Proposition 
    \ref{invIsPtCount}.
\end{proof}

\begin{remark}
    From this result, we see that the intersections we 
    perform on $\LogPk$ can be arranged to avoid the boundary. 
    The fact that the alternative compactification 
    $\LogPkint  \subset \PP^{n-3} \times \PP^k$ gives us 
    an \textit{upper bound} and not an exact calculation of the 
    invariants tells us that here, intersections 
    \textit{cannot} be made to avoid the boundary. The reason 
    for this is that the evaluation map $\PP^{n-3} \times 
    \PP^k \mathrel{-\,}\rightarrow \prod_i \PP^k$ is only a rational map. 
    Intersections can therefore be made to avoid the boundary 
    of a blowup, but this will in general not have the 
    positivity we need.
\end{remark}

\subsection{Calculations}\label{calculations}

We turn to direct computations. The reader may benefit from 
bearing in mind that, since we want to find an upper bound
for the degree of a divisor in a product of projective 
spaces, it is enough to find any polynomial vanishing 
along the divisor and read off its bidegree. 
In the remainder of this Section, $H$ will denote a generic 
hyperplane in $\PP^k$ and we will write 
\[A_{n+k-4}(\PP^{n-3} \times \PP^k) = \Z H_1 \oplus \Z H_2.\]
We first deal with the easy case of evaluation at $p=p_{i_0}.$

\begin{lemma} \label{easyEv}
    We have 
    $[\overline{\ev_{i_0}^{-1} H} ] = H_2.$
\end{lemma}

\begin{proof}
    Under our chosen isomorphism $\operatorname{fgt} 
    \times \ev_{i_0} : \LogPkint \to \mathcal{M}_{0,n} 
    \times T,$ the map $\ev_{i_0}$ is identified with 
    projection onto the second factor.
\end{proof}

\begin{proposition} \label{classbound}
    For $i \neq i_0$, one has $[\overline{\ev_i^{-1}H}] = \upchi H_1
    + H_2$ with 
    \[\upchi \leq \sum_{i=1}^n \upalpha^{\max}_i,\] where 
    $\upalpha^{\max}_i = \max_j \upalpha_{ij}.$
\end{proposition}

\begin{proof}
Suppose for now that $i \neq 1,2,3,$ and that $p_i$ has no 
tangency to $\partial \PP^k$. This way we are not 
evaluating at a marked point whose position in $\PP^1$ is 
fixed, and the evaluation map factors as
\[\ev_i: \mathcal{M}_{0,n} \times T \to T  \hookrightarrow 
\PP^k.\]
As discussed in Section \ref{logGWbackground}, any $f \in \LogPkint$ 
is determined up to units by
\begin{equation}\label{eq2} f_j \colon [X:Y] \mapsto
\prod_{i=1}^n (X \, T_i - Y \, S_i)^{\upalpha_{ij}}, \quad j = 0, \dots, k, \end{equation}
recalling that $[S_1:T_1], [S_2:T_2], [S_3:T_3]$ have fixed
values. In summary,
\[ f \colon [X:Y] \mapsto \big[\uplambda_0 f_0
\left([X:Y] \right) : \uplambda_1 f_1 \left( [X:Y] \right) : 
\cdots : \uplambda_k f_k \left([X:Y] \right) \big], \] 
for some $\uplambda_j \in \CC^{\times}.$
Because $p$ has zero contact with the boundary, the value of $f(p)$ and the tuple $(\uplambda_j)$ are equivalent data. 
Computing the other evaluation maps is achieved by passing back and forth between this equivalence. Precisely, a point 
$(\{ [S_i:T_i] \}_i, [A_0: \dots : A_k]) \in \mathcal{M}_{0,n} \times T$ has 
\[ [A_0 : \dots : A_k] = \left[\uplambda_0 f_0(p) : \uplambda_1 f_1(p) : \dots : \uplambda_k f_k(p) \right], \]
so \[ \frac{\uplambda_j}{\uplambda_0} = \frac{A_j}{A_0} \frac{f_0(p)}{f_j(p)}. \]
We deduce that $\ev_i : \mathcal{M}_{0,n} \times T \longrightarrow T$ sends
\[ \left( \{[S_i:T_i] \}_i, [A_0: \cdots A_k] \right) \mapsto \left[ 
f_0(p_i) : \frac{A_1}{A_0} \frac{f_0(p)}{f_1(p)} f_1(p_i):
\dots : \frac{A_k}{A_0} \frac{f_0(p)}{f_k(p)} f_k(p_i)\right]. \]
While the $T$ in the domain and codomain are the same torus, it
will help to give them distinct coordinates. On the target 
torus $T,$ let $B_0, \dots, B_k$
be the restriction of the homogeneous coordinates on $\PP^k.$
We can then pick a generic hyperplane
$H \subset T$ defined by \[ H = \left\{ \sum_{j=0}^k \upmu_j B_j = 0 \right\}. \]
We have
\[ \ev_i^{-1} H = \left \{ \sum_{j=0}^k \upmu_j A_j \frac{f_0(p)}
{f_j(p)} f_j(p_i) = 0 \right \}.\]
Clearing denominators expresses $\ev_i^{-1} H$ as
\[ \left\{ F_i(p_1,\dots,p_n) \coloneq \sum_{j=0}^k 
\left(\upmu_j A_j f_j(p_i) \prod_{l\neq j} f_l(p) 
\right) = 0 \right\}.\] 
From \eqref{eq2} and our choices of coordinates for $p_1, 
p_2, p_3,$ we see that 
\[f_j(p_i) = S_i^{\upalpha_{1j}} (S_i-T_i)^{\upalpha_{2j}} 
T_i^{\upalpha_{3j}} \prod_{u=4}^n 
(S_i T_u - T_i S_u)^{\upalpha_{uj}},\]
and 
\[f_l(p) = S_{i_0}^{\upalpha_{1l}} (S_{i_0} - 
T_{i_0})^{\upalpha_{2l}} T_{i_0}^{\upalpha_{3l}} \prod_{u=4}^n
(S_{i_0} T_u - T_{i_0} S_u)^{\upalpha_{ul}}.\]
In terms of the affine coordinates \eqref{eq1} on $\mathcal{M}_{0,n}
\times T$ we have: 
\[ \ev_i^{-1}H = \left\{ \sum_{j=0}^k \upmu_j a_j \left(\uptheta_{ij}
\prod_{u=4}^n (x_i - x_u)^{\upalpha_{uj}} \right)
\left( \upphi_{ij} \prod_{v=4}^n (x_{i_0} - x_v)^
{\sum_{l \neq j} \upalpha_{vl}} \right) = 0 \right\},\]
where
\[\uptheta_{ij} = x_i^{\upalpha_{1j}} 
(x_i-1)^{\upalpha_{2j}} (-1)^{\upalpha_{3j}}, \quad 
\upphi_{ij} = x_{i_0}^{\sum_{l \neq j} \upalpha_{1l}}
(x_{i_0} -1)^{\sum_{l \neq j} \upalpha_{2l}} 
(-1)^{\sum_{l \neq j} \upalpha_{3l}},\] and $a_0=1.$
Note that $F_i$ is \textit{not} the zero 
polynomial since $x_i$ and $x_{i_0}$ were assumed not to meet
the boundary $\partial \PP^k,$ so 
$\upalpha_{ij}, \upalpha_{i_0 j} = 0$ for all $j.$

The equation defining $\overline{\ev_i^{-1}H} \subset 
\PP^{n-3} \times \PP^k$ certainly divides
\[\tilde{F}_i = \sum_{j=0}^k \upmu_j A_j \left(\Uptheta_{ij}
\prod_{u=4}^n (X_i - X_u)^{\upalpha_{uj}} \right)
\left( \Upphi_{ij}
\prod_{v=4}^n (X_{i_0} - X_v)^{\sum_{l \neq j} \upalpha_{vl}} \right),\]
where
\[\Uptheta_{ij} = X_i^{\upalpha_{1j}}(X_i-X_0)^{\upalpha_{2j}}
(-X_0)^{\upalpha_{3j}}, \quad \Upphi_{ij} = X_{i_0}^
{\sum_{l \neq j} \upalpha_{1l}} (X_{i_0} - X_0)^
{\sum_{l \neq j} \upalpha_{2j}} (-X_0)^{\sum_{l \neq j} 
\upalpha_{3j}}.\]
The balancing condition shows that $\tilde{F}_i$ is of
bidegree $(dn,1).$ Indeed
\[\sum_{i=1}^n \upalpha_{ij} + \sum_{i=1}^n\sum_{l \neq j} \upalpha_{il}
= \sum_{i,j} \upalpha_{ij} = dn.\]
But $\tilde{F}_j$ is
reducible, with some factors cutting out parts 
of the boundary of $\LogPkint$ inside $\PP^{n-3} \times 
\PP^k.$ Our previous estimate of the bidegree of $[\overline{
\ev_i^{-1}H}],$ namely $(1, dn),$ can therefore be
improved upon. 

Since no marked point 
can be tangent to every component of the boundary we will not 
find common factors from the first bracketed factor. However, 
removing common factors from the second set of brackets
leaves a polynomial of degree 
\[dn - \sum_{i=1}^n \min_j \left( \sum_{l \neq j} \upalpha_{il} 
\right) = dn - \sum_{i=1}^n \left( \sum_{j=1}^k \upalpha_{ij} - 
\upalpha_i^{\max} \right) = \sum_i \upalpha_i^{\max}\]
in the $X_i$ as claimed. Note that the factors we cancel 
are all invertible on $\LogPkint$ where marked points 
cannot collide, so what remains still contains $\ev_i^{-1}H.$

We conclude by explaining how to do away with the 
simplifying assumptions made in the first line of the proof.
In the case of evaluating at $p_1, p_2, p_3,$ the 
$\uptheta_{ij}$ which were previously of positive degree 
become constant. The bound is less sharp in these cases,
but nevertheless true. Meanwhile, if $p_i$ has tangency 
to the boundary, then $\ev_i$ instead has image in the dense
torus of the stratum to which $p_i$ maps. This is a quotient
of the torus $T,$ so the same argument goes through with 
fewer terms in $F_i$. Again this introduces some 
inefficiency but the bound holds nonetheless.
\end{proof}

\begin{proof}[Proof of Theorem \ref{Main}.]

Lemma \ref{easyEv} and Propositions \ref{mainTheory} and 
\ref{classbound} combine to give \[ N_{\upalpha, \upnu} \leq \deg H_2^{\upnu_{i_0}} \prod_{i \neq i_0} \left(\left(\sum_i \upalpha_i^{\max}\right)H_1 + H_2 \right)^{\upnu_i}.\]
Extracting the coefficient of $H_1^{n-3} H_2^k$ in the above yields
\[N_{\upalpha, \upnu} \leq \binom{k-\upnu_{i_0}+n-3}{n-3} \left( \sum_i \upalpha_i^{\max} \right)^{n-3},\]
recalling that $\sum_i \upnu_i = k+n-3.$ Minimising over the possible
$i_0$ gives the result.
\end{proof}
\section{Examples}

We conclude with some comparisons between our upper bound 
and known invariants of $\PP^2$. The most salient feature is that the 
bound becomes less sharp as the number of marked points increases. All of
our exact calculations use Mikhalkin's tropical correspondence 
theorem \cite{mikhalkin2004enumerativetropicalalgebraicgeometry}.
We assume some familiarity with tropical curves.

In all of our examples, we display the markings with nonzero 
tangencies to the boundary $\partial \PP^2.$ At each of the 
remaining markings we impose a point constraint. A virtual 
dimension check shows that, if we have $n'$ markings with 
positive tangency, there should be $n'-1$ markings with no 
tangencies carrying point constraints. Therefore there are 
$n = 2n'-1$ markings in total here.

\begin{example}
    We begin with the case of maximal contact discussed in the
    Introduction. 
    \begin{table}[H]
        \centering
        \begin{tabular}{c|c c c}
             & $\, H_1 \,$ & $\, H_2 \,$ & $\, H_3 \,$ \\ \hline
             $x_1$ & $d$ & 0 &  0\\ 
             $x_2$ & 0 &  $d$ & 0 \\ 
             $x_3$ & 0 & 0 & $d$ \\ 
        \end{tabular}
        \label{tangencies1}
    \end{table}
    There is only one tropical curve with three legs, as 
    pictured below. 
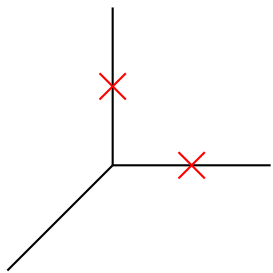
\begin{figure}[H]
\begin{tikzpicture}[scale=0.7]
 \draw[thick] (0,0) -- (3,0);        
 \draw[thick] (0,0) -- (0,3);       
 \draw[thick] (0,0) -- (-2,-2);      

 \draw[thick, red] (-0.25,1.25) -- (0.25,1.75);
 \draw[thick, red] (-0.25,1.75) -- (0.25, 1.25);
 \draw[thick, red] (1.25,-0.25) -- (1.75,0.25);
 \draw[thick, red] (1.75,-0.25) -- (1.25,0.25);
\end{tikzpicture}
\end{figure}
    This contributes a multiplicity of 
    $\det \begin{pmatrix}
        d & 0 \\ 0 & d
    \end{pmatrix} = d^2,$ so in this case the exact invariant 
    is $d^2.$

    Meanwhile, the upper bound is $(d+d+d)^{(2\cdot3 -1)-3} = 
    9d^2.$ In this case the bound is off only by a constant 
    factor.
\end{example}

In fact, whenever only three points carry positive tangency 
with respect to the boundary the simple analysis above gives 
way to a general argument. One can therefore check the upper
bound by hand in this case. 

\begin{proposition}
    Theorem \ref{Main} is true when $k=2$ and exactly 3 
    marked points are tangent to $\partial \PP^2.$
\end{proposition}

\begin{proof}
    Label the tangencies as follows 
    \begin{table}[H]
        \centering
        \begin{tabular}{c|c c c}
             & $\, H_1 \,$ & $\, H_2 \,$ & $\, H_3 \,$ \\ \hline
             $x_1$ & $a_1$ & $a_2$ & $a_3$ \\ 
             $x_2$ & $b_1$ & $b_2$ & $b_3$ \\ 
             $x_3$ & $c_1$ & $c_2$ & $c_3$ \\ 
        \end{tabular}
        \label{tangencies2}
    \end{table}
    Note that at least one entry per row is 0 since the 
    intersection of the $H_j$ is empty.
    There is a unique tropical curve with these data, 
    contributing a multiplicity of 
    \[N = \det \begin{pmatrix} a_1-a_3 & b_1-b_3\\ a_2-a_3 & b_2 - b_3 \end{pmatrix} = \det \begin{pmatrix}
    b_1-b_3 & c_1 - c_3 \\ b_2 - b_3 & c_2 - c_3 \end{pmatrix} = \det \begin{pmatrix} c_1 - c_3
    & a_1 - a_3 \\ c_2 - c_3 & a_2 - a_3 \end{pmatrix}\]
    to the tropical count.
    There is only one way to pass it through two generic
    points in $\RR^2,$ so the invariant is 
    \begin{multline} \label{eq3}
    \frac13 \big[ (a_1 - a_3)(b_2-b_3) - (a_2-a_3) (b_1-b_3) + \\
    (b_1-b_3)(c_2-c_3) - (b_2-b_3)(c_1-c_3)
    + \\ (c_1 - c_3)(a_2-a_3) - (c_2-c_3) (a_1-a_3) \big].
    \end{multline}
    Meanwhile the upper bound is
    \begin{equation} \label{eq4}
    \left( a_{\max} + b_{\max} + c_{\max}\right)^2.
    \end{equation}
    That the former is bounded above by the latter is then a 
    quick piece of algebra; after expanding out 
    \eqref{eq3} and removing the negative terms, and
    replacing $a_i, b_i, c_i$ by their respective maxima, 
    the claimed inequality 
    \eqref{eq3} $\leq$ \eqref{eq4} is implied by 
    \[\frac43 (a_{\max} b_{\max} + b_{\max} c_{\max} +
    c_{\max} a_{\max}) \leq (a_{\max} + b_{\max} + 
    c_{\max})^2. \qedhere \] 
\end{proof}

\begin{example}
We now look at an example where four marked points are tangent to
the toric boundary.
    \begin{table}[H]
        \centering
        \begin{tabular}{c|c c c}
             & $\, H_1 \,$ & $\, H_2 \,$ & $\, H_3 \,$ \\ \hline
             $x_1$ & $d$ & 0 & 0 \\ 
             $x_2$ & 0 & $d-1$ & 0 \\ 
             $x_3$ & 0 & 1 & 0 \\
             $x_4$ & 0 & 0 & $d$
        \end{tabular}
        \label{tangencies3}
    \end{table}
    The tropical curves which contribute are the following. 
    
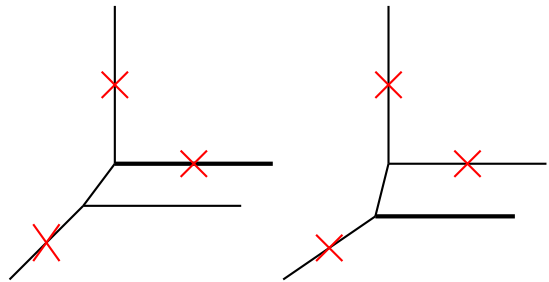
\begin{figure}[H]
        \centering
        \begin{tikzpicture}[scale = 0.7]
       \draw[ultra thick] (0,0) -- (3,0);
       \draw[thick] (0,0) -- (0,3);
       \draw[thick] (0,0) -- (-0.6,-0.8);
       \draw[thick] (-0.6,-0.8) -- (2.4,-0.8);
       \draw[thick] (-0.6,-0.8) -- (-2,-2.2);
        \draw[thick, red] (-0.25,1.25) -- (0.25,1.75);
        \draw[thick, red] (-0.25,1.75) -- (0.25, 1.25);
        \draw[thick, red] (1.25,-0.25) -- (1.75,0.25);
        \draw[thick, red] (1.75,-0.25) -- (1.25,0.25);
       \draw[thick,red] (-1.3-0.25, -1.5-0.35) --
       (-1.3+0.25,-1.5+0.35);
       \draw[thick,red] (-1.3-0.25,-1.5+0.35) --
       (-1.3+0.25,-1.5-0.35);
   \end{tikzpicture}
   \begin{tikzpicture}[scale = 0.7]
       \draw[thick] (0,0) -- (3,0);
       \draw[thick] (0,0) -- (0,3);
       \draw[thick] (0,0) -- (-0.25,-1);
       \draw[ultra thick] (-0.25,-1) -- (2.4,-1);
       \draw[thick] (-0.25,-1) -- (-2,-2.2);
        \draw[thick, red] (-0.25,1.25) -- (0.25,1.75);
        \draw[thick, red] (-0.25,1.75) -- (0.25, 1.25);
        \draw[thick, red] (1.25,-0.25) -- (1.75,0.25);
        \draw[thick, red] (1.75,-0.25) -- (1.25,0.25);
        \draw[thick,red] (-1.125-0.25, -1.6-0.25) --
       (-1.125+0.25,-1.6+0.25);
       \draw[thick,red] (-1.125-0.25,-1.6+0.25) --
       (-1.125+0.25,-1.6-0.25);
   \end{tikzpicture}
    \end{figure}
    Note that the finite edges have different slopes, 
    accounting for the fact that the two horizontal edges
    have weights 1 and $d-1$, distinguished in the diagram
    by the thickness of the ray. Each tropical curve 
    is weighted with multiplicity $d \cdot d(d-1)$, and can be 
    passed through 3 generic points in $\RR^2$ in a unique way.
    The invariant is therefore \[d^2(d-1) + d^2 (d-1) = 2d^2(d-1).\]
    
    Meanwhile, the upper bound is $(d+d+(d-1)+1)^4 = 81d^4$.
\end{example}

\begin{example}
Finally, we discuss another example with four points of tangency
to the boundary, but with three points incident to one of the 
components.
    \begin{table}[H]
        \centering
        \begin{tabular}{c|c c c}
             & $\, H_1 \,$ & $\, H_2 \,$ & $\, H_3 \,$ \\ \hline
             $x_1$ & $d$ & $d-2$ & 0 \\ 
             $x_2$ & 0 & 1 & 0 \\ 
             $x_3$ & 0 & 1 & 0 \\
             $x_4$ & 0 & 0 & $d$
        \end{tabular}
        \label{tangencies4}
    \end{table}
    The only tropical curve contributing to the invariant is 

    
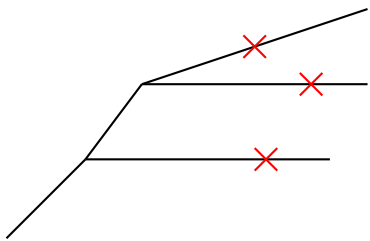
\begin{figure}[H]
    \centering
    \begin{tikzpicture}[scale = 1.5]
       \draw[thick] (0,0) -- (2,0.6666);
       \draw[thick] (0,0) -- (2,0);
       \draw[thick] (0,0) -- (-0.5, -0.6666);
       \draw[thick] (-0.5, -0.6666) -- (1.6666, -0.6666);
       \draw[thick] (-0.5, -0.6666) -- (-0.5-0.7, -0.6666-0.7);
       \draw[thick,red] (1-0.1,0.3333-0.1) --
       (1+0.1,0.3333+0.1);
       \draw[thick,red] (1-0.1,0.3333+0.1) --
       (1+0.1,0.3333-0.1);
       \draw[thick,red] (1.5-0.1,-0.1) -- (1.5+0.1,+0.1);
       \draw[thick,red] (1.5-0.1,+0.1) -- (1.5+0.1,-0.1);
       \draw[thick,red] (1.1-0.1,-0.6666-0.1) --
       (1.1+0.1,-0.6666+0.1);
       \draw[thick,red] (1.1-0.1,-0.6666+0.1) --
       (1.1+0.1,-0.6666-0.1);
   \end{tikzpicture}
    \end{figure}
    This has multiplicity $d\cdot d$, and can be passed through 
    3 generic points in $\RR^2$ in a unique way.
    Therefore the actual invariant is $d^2$. 
    
    By contrast, the upper bound is $(2d+2)^4$.
\end{example}

\printbibliography

@article{Ranganathan_2017,
   title={Skeletons of Stable Maps I: Rational Curves in Toric Varieties},
   volume={95},
   number={3},
   journal={Journal of the London Mathematical Society},
   publisher={Wiley},
   author={Ranganathan, Dhruv},
   year={2017},
   month=feb, 
   pages={804–832} }

@article {mikhalkin2004enumerativetropicalalgebraicgeometry,
    AUTHOR = {Mikhalkin, Grigory},
     TITLE = {Enumerative tropical algebraic geometry in {$\RR^2$}},
   JOURNAL = {J. Amer. Math. Soc.},
  FJOURNAL = {Journal of the American Mathematical Society},
    VOLUME = {18},
      YEAR = {2005},
    NUMBER = {2},
     PAGES = {313--377},
   MRCLASS = {14N10 (05A99 14N35 52B70)},
  MRNUMBER = {2137980},
MRREVIEWER = {Charles\ D.\ Cadman},
    
}

@article {GS2013logarithmic,
    AUTHOR = {Gross, Mark and Siebert, Bernd},
     TITLE = {Logarithmic {G}romov-{W}itten invariants},
   JOURNAL = {J. Amer. Math. Soc.},
  FJOURNAL = {Journal of the American Mathematical Society},
    VOLUME = {26},
      YEAR = {2013},
    NUMBER = {2},
     PAGES = {451--510},
   MRCLASS = {14N35 (14D23)},
  MRNUMBER = {3011419},
MRREVIEWER = {Hsian-Hua\ Tseng},
}

@book {FultonIntersection,
    AUTHOR = {Fulton, William},
     TITLE = {Intersection theory},
    SERIES = {Ergebnisse der Mathematik und ihrer Grenzgebiete. 3. Folge. A
              Series of Modern Surveys in Mathematics [Results in
              Mathematics and Related Areas. 3rd Series. A Series of Modern
              Surveys in Mathematics]},
    VOLUME = {2},
   EDITION = {Second},
 PUBLISHER = {Springer-Verlag, Berlin},
      YEAR = {1998},
     PAGES = {xiv+470},
   MRCLASS = {14C17 (14-02)},
  MRNUMBER = {1644323},
}

@article {AbramChen,
    AUTHOR = {Abramovich, Dan and Chen, Qile},
     TITLE = {Stable logarithmic maps to {D}eligne-{F}altings pairs {II}},
   JOURNAL = {Asian J. Math.},
  FJOURNAL = {Asian Journal of Mathematics},
    VOLUME = {18},
      YEAR = {2014},
    NUMBER = {3},
     PAGES = {465--488},
   MRCLASS = {14D23 (14A20 14H10 14N35)},
  MRNUMBER = {3257836},
MRREVIEWER = {Jonathan\ Wise},
}

@article {Chen,
    AUTHOR = {Chen, Qile},
     TITLE = {Stable logarithmic maps to {D}eligne-{F}altings pairs {I}},
   JOURNAL = {Ann. of Math. (2)},
  FJOURNAL = {Annals of Mathematics. Second Series},
    VOLUME = {180},
      YEAR = {2014},
    NUMBER = {2},
     PAGES = {455--521},
   MRCLASS = {14N35 (14D23)},
  MRNUMBER = {3224717},
MRREVIEWER = {Sergiy\ Koshkin},
}

@incollection {KontsevichManin,
    AUTHOR = {Kontsevich, M. and Manin, Yu.},
     TITLE = {Gromov-{W}itten classes, quantum cohomology, and enumerative
              geometry [{MR}1291244 (95i:14049)]},
 BOOKTITLE = {Mirror symmetry, {II}},
    SERIES = {AMS/IP Stud. Adv. Math.},
    VOLUME = {1},
     PAGES = {607--653},
 PUBLISHER = {Amer. Math. Soc., Providence, RI},
      YEAR = {1997},
   MRCLASS = {14N10 (53C15 58A10 58F05)},
  MRNUMBER = {1416351},
}

@misc{FNSPositivity,
      title={Counting point configurations in projective space}, 
      author={Alex Fink and Navid Nabijou and Rob Silversmith},
      year={2026},
      eprint={2601.15421},
      archivePrefix={arXiv},
      primaryClass={math.AG},
}

@article {FultonSturmfels,
    AUTHOR = {Fulton, William and Sturmfels, Bernd},
     TITLE = {Intersection theory on toric varieties},
   JOURNAL = {Topology},
  FJOURNAL = {Topology. An International Journal of Mathematics},
    VOLUME = {36},
      YEAR = {1997},
    NUMBER = {2},
     PAGES = {335--353},
   MRCLASS = {14M25 (14C17 52B20)},
  MRNUMBER = {1415592},
MRREVIEWER = {Michel\ Brion},
}

@article {CoxHomog,
    AUTHOR = {Cox, David A.},
     TITLE = {The homogeneous coordinate ring of a toric variety},
   JOURNAL = {J. Algebraic Geom.},
  FJOURNAL = {Journal of Algebraic Geometry},
    VOLUME = {4},
      YEAR = {1995},
    NUMBER = {1},
     PAGES = {17--50},
   MRCLASS = {14M25},
  MRNUMBER = {1299003},
MRREVIEWER = {Mina\ Teicher},
}

@article {BELL,
    AUTHOR = {Brakensiek, Joshua and Eur, Christopher and Larson, Matt and
              Li, Shiyue},
     TITLE = {Kapranov degrees},
   JOURNAL = {Int. Math. Res. Not. IMRN},
  FJOURNAL = {International Mathematics Research Notices. IMRN},
      YEAR = {2025},
    NUMBER = {20},
     PAGES = {Paper No. rnaf306, 16},
   MRCLASS = {14H10 (05E14)},
  MRNUMBER = {4970179},
}

@article {Kleiman,
    AUTHOR = {Kleiman, Steven L.},
     TITLE = {The transversality of a general translate},
   JOURNAL = {Compositio Math.},
  FJOURNAL = {Compositio Mathematica},
    VOLUME = {28},
      YEAR = {1974},
     PAGES = {287--297},
   MRCLASS = {14M15 (14N10)},
  MRNUMBER = {360616},
MRREVIEWER = {Joel\ Roberts},
}

@article {AbramovichWise,
    AUTHOR = {Abramovich, Dan and Wise, Jonathan},
     TITLE = {Birational invariance in logarithmic {G}romov-{W}itten theory},
   JOURNAL = {Compos. Math.},
  FJOURNAL = {Compositio Mathematica},
    VOLUME = {154},
      YEAR = {2018},
    NUMBER = {3},
     PAGES = {595--620},
   MRCLASS = {14N35 (14A20 14D23 14E05 14H10)},
  MRNUMBER = {3778185},
MRREVIEWER = {Hsian-Hua\ Tseng},
}
\end{document}